\documentclass[11pt]{article}
\usepackage[margin=3cm]{geometry}
\usepackage[dvipsnames]{xcolor}
\usepackage{amsmath}
\usepackage{amsthm}
\usepackage{tikz}
\usepackage{algorithm}
\usepackage{algpseudocode}

\usepackage{amssymb}
\usetikzlibrary{decorations.markings}
\usepackage{bbold}
\newtheorem*{rep@theorem}{\rep@title}
\newcommand{\newreptheorem}[2]{
\newenvironment{rep#1}[1]{
 \def\rep@title{#2 \ref{##1}}
 \begin{rep@theorem}}
 {\end{rep@theorem}}}
\makeatother
\newtheorem{theorem}{Theorem}
\newtheorem{lemma}[theorem]{Lemma}
\newreptheorem{lemma}{Lemma}

\newtheorem{corollary}[theorem]{Corollary}

\theoremstyle{definition}
\newtheorem{definition}[theorem]{Definition}

\newcommand{\Lalpha}{\Lscr_1^{> \alpha}}

\newcommand{\Cscr}{\mathcal{C}}
\newcommand{\Bscr}{\mathcal{B}}
\newcommand{\inside}{\subseteq}
\newcommand{\Nscr}{\mathcal{N}}
\newcommand{\Lscr}{\mathcal{L}}
\newcommand{\Fscr}{\mathcal{F}}
\newcommand{\Iscr}{\mathcal{I}}
\newcommand{\Pscr}{\mathcal{P}}

\definecolor{darkgreen}{rgb}{0.0, 0.8, 0.0}
\begin{document}

\title{An improved integrality gap for disjoint cycles in planar graphs}
\author{Niklas Schlomberg\thanks{Research Institute for Discrete Mathematics and Hausdorff Center for Mathematics, University of Bonn}}
\date{\vspace*{-8mm}}
\maketitle

\begin{abstract}
We present a new greedy rounding algorithm for the Cycle Packing Problem for uncrossable cycle families in planar graphs.
This improves the best-known upper bound for the integrality gap of the natural packing LP to a constant slightly less than $3.5$.
Furthermore, the analysis works for both edge- and vertex-disjoint packing.
The previously best-known constants were $4$ for edge-disjoint and $5$ for vertex-disjoint cycle packing.

This result also immediately yields an improved Erd\H{o}s--P\'osa ratio:
for any uncrossable cycle family in a planar graph, the minimum number of vertices (edges) needed to hit all cycles in the family is less than 8.38 times the maximum number of vertex-disjoint (edge-disjoint, respectively) cycles in the family.

Some uncrossable cycle families of interest to which the result can be applied are the family of all cycles in a directed or undirected graph, in undirected graphs also the family of all odd cycles and the family of all cycles containing exactly one edge from a specified set of demand edges.
The last example is an equivalent formulation of the fully planar Disjoint Paths Problem.
Here the Erd\H{o}s--P\'osa ratio translates to a ratio between integral multi-commodity flows and minimum cuts.
\end{abstract}

\section{Introduction}

Given a family $\Cscr$ of cycles in a (directed or undirected) graph $G$, the Cycle Packing Problem asks for a maximum-cardinality subset $\Cscr^* \subseteq \Cscr$ of pairwise vertex- or edge-disjoint cycles.
It admits the natural packing LP
\begin{equation}\label{eq:lp}
\max \left\{ \sum_{C\in\Cscr}x_C : \sum_{C\in\Cscr: v\in C}x_C\le 1 \ (v \in V),\ x_C\ge 0 \ (C\in\Cscr) \right\}
\end{equation}
for vertex-disjoint cycle packing and
\begin{equation}\label{eq:lp_edgedisjoint}
\max \left\{ \sum_{C\in\Cscr}x_C : \sum_{C\in\Cscr: e\in C}x_C\le 1 \ (e\in E),\ x_C\ge 0 \ (C\in\Cscr) \right\}
\end{equation}
for edge-disjoint cycle packing.
Despite its exponentially many variables, optimum LP solutions can be computed in polynomial time if $\Cscr$ is given by a \emph{weight oracle}~\cite{SchTV22}:

\begin{definition}[\cite{SchTV22}]
 Let $\Cscr$ be a family of cycles in a graph $G$.
 $\Cscr$ has a \emph{weight oracle} if for any edge weights $w \colon E(G) \to \mathbb{R}_{\geq 0}$ we can compute a weight-minimal cycle in $\Cscr$ in polynomial time.
\end{definition}

For arbitrary graphs the integrality gap of the LPs~\eqref{eq:lp} and~\eqref{eq:lp_edgedisjoint} is unbounded even if $\Cscr$ is the set of all odd cycles in $G$~\cite{Ree99,RauR01}.
For planar graphs however, Schlomberg, Thiele and Vygen~\cite{SchTV22} have recently shown constant upper bounds for the integrality gaps if the cycle family $\Cscr$ is \emph{uncrossable}.

\begin{definition}[Goemans, Williamson~\cite{GoeW98}] \label{def:uncrossable}
A family $\Cscr$ of cycles in a graph is called \emph{uncrossable} if the following property holds.

Let $C_1,C_2\in\Cscr$ and let $P_2$ be a path in $C_2$ such that $P_2$ shares only its endpoints with $C_1$. 
Then there is a path $P_1$ in $C_1$ between these endpoints such that
$P_1+P_2\in\Cscr$ and $(C_1 - P_1)+(C_2 -P_2)$ contains a cycle in $\Cscr$ (as an edge set).
\end{definition}

In their work they give an upper bound of $5$ for the vertex-disjoint cycle packing LP~\eqref{eq:lp}, using a greedy rounding algorithm.
For the edge-disjoint LP~\eqref{eq:lp_edgedisjoint} they show an upper bound of $4$, generalizing a similar result for the edge-disjoint paths problem by Garg, Kumar and Seb\H{o} \cite{GarKS22}.
In this work we modify their greedy rounding algorithm and analyze it using a new structural lemma. This improves the integrality gaps of both LPs to below $3.5$:

\begin{theorem}
  Let $G$ be a planar graph, embedded in the sphere, and $\Cscr$ an uncrossable family of cycles in $G$.
  Then there exists an integral solution to the vertex- or edge-disjoint cycle packing LP~\eqref{eq:lp} or~\eqref{eq:lp_edgedisjoint} with at least $\frac{9}{20 + \sqrt{130}} > \frac{1}{3.5}$ the LP value.
  If $\Cscr$ is given by a weight oracle we can compute such a solution in polynomial time.
 \end{theorem}
 
 \subsection{The Erd\H{o}s--P\'osa ratio}
 
 The duals of the LPs~\eqref{eq:lp} and~\eqref{eq:lp_edgedisjoint} are relaxations of the Cycle Transversal Problem:
 This asks for a minimum subset of vertices or edges, respectively, that hit each cycle in $\Cscr$.
 Berman and Yaroslavtsev~\cite{BerY12} have shown an upper bound of $2.4$ for the integrality gaps of the edge and vertex cycle transversal LPs, improving on a previous bound of $3$ by Goemans and Williamson~\cite{GoeW98}.
 Multiplying the integrality gaps of the primal and dual LPs directly yields a maximum ratio between integral solutions to the primal and the dual:
 
 \begin{corollary}
  Let $G$ be a planar graph and $\Cscr$ an uncrossable family of cycles in $G$.
  Let $\nu_v$ respectively $\nu_e$ be the maximum number of vertex- and edge-disjoint cycles in $\Cscr$
  and let $\tau_v$ respectively $\tau_e$ be the minimum size of vertex and edge transversals for $\Cscr$. \newline
  Then $\frac{\tau_v}{\nu_v} \leq 2.4 \cdot \frac{20 + \sqrt{130}}{9} \leq 8.38$ and $\frac{\tau_e}{\nu_e} \leq 2.4 \cdot \frac{20 + \sqrt{130}}{9} \leq 8.38$
 \end{corollary}
 
 The supremum of $\frac{\tau_v}{\nu_v}$ respectively $\frac{\tau_e}{\nu_e}$ is known as the Erd\H{o}s--P\'osa ratio for the Cycle Packing and Transversal Problem.
 The previously best-known upper bound for general uncrossable cycle families was $12$ for vertex-disjoint Cycle Packing and $9.6$ for edge-disjoint Cycle Packing~\cite{SchTV22};
 both resulting from multiplying the upper bounds for the primal and dual integrality gaps.
 
 \subsection{New Techniques}
 
 Our main algorithm that we use to find integral solutions to the Cycle Packing LP with the claimed approximation guarantee is similar to the greedy rounding algorithm used in~\cite{SchTV22}:
 Similar to~\cite{SchTV22}, we start by solving the LP and applying an uncrossing procedure to obtain an optimum LP solution where the cycles in the support form a \emph{laminar} family $\Lscr$.
 Afterwards, we iteratively pick a set $\Fscr^* \subseteq \Lscr$ of pairwise disjoint cycles, add them to our solution and remove all ``neighbours'', i.e.\@ cycles that are not disjoint to $\Fscr^*$, from the support.
 
 Schlomberg, Thiele and Vygen \cite{SchTV22} showed that they can always find a single cycle $C^*$ with LP value at most $5$ on its neighbourhood.
 In this work we prove a new Structure Lemma, showing that after a slight modification of our LP solution, for a \emph{one-sided} cycle $C$ (i.e.\@ a cycle with a side that contains no other cycles in $\Lscr$) the average LP value on the neighbours of $C$ without $C$ itself is at most $3$.
 
 In Section~\ref{sec:bounding_the_gap} we observe that this already improves the bound of $5$ from \cite{SchTV22} to $4$.
 However, by exploiting that we can also add several cycles in a single iteration to our solution, we can improve the bound further to below $3.5$ (see Section~\ref{sec:improving_the_bounds}).
 To this end, we add further candidates for our set $\Fscr^*$:
 First, all one-sided cycles with LP value at least $\alpha > \frac{1}{2}$ are pairwise disjoint.
 Second, among the one-sided cycles with LP value at least $\alpha > \frac{1}{4}$ we can find a large set of pairwise disjoint cycles due to the Four Colour Theorem. 
 
 Key of the new results is the Structure Lemma~\ref{lemma:structure_lemma}, which we prove by constructing a set $M^*$ of LP constraints (i.e.\@ vertices or edges) that cover each cycle in $\Lscr$ enough often.
 Section~\ref{sec:proof_of_main_lemma} is dedicated to the construction of $M^*$.
 We proceed by an induction on the number of cycles in $\Lscr$.
 If all cycles are one-sided, an auxiliary graph, similar to the planar dual, directly yields a feasible set $M^*$.
 Otherwise, we pick a minimal two-sided cycle, find feasible sets $M^*$ for both sides of it and carefully put them together to a solution for the whole family.
 
 \subsection{Examples for uncrossable cycle families}
 
 There are several examples for cycle families in $G$ that are always uncrossable and many of them have been studied individually.
 A list of the most interesting examples together with proofs of their uncrossability can be found in~\cite{SchTV22}.
 
 The first example of interest is the set of all cycles in an undirected graph $G$.
 For this problem Erd\H{o}s and P\'osa~\cite{ErdP65} showed that even in general, not necessarily planar, graphs with bounded cycle packing number the transversal number is bounded, although in general the ratio is unbounded.
 This property is known as the Erd\H{o}s--P\'osa property.
 In planar graphs the Erd\H{o}s--P\'osa ratio is $4$ for edge-disjoint packing (the upper bound comes from a result by Ma, Yu and Zang~\cite{MaYZ13}, tightness was shown by an example by Kr\'al\textquoteright{}~\cite{MaYZ13}).
 For vertex-disjoint packing~\cite{CheFS12} and~\cite{MaYZ13} gave an upper bound of $3$ on the Erd\H{o}s--P\'osa ratio.
 
 Also the set of all directed cycles in a digraph $G$ is uncrossable.
 Here again the Erd\H{o}s--P\'osa property holds on arbitrary graphs \cite{ReeRST96}.
 For planar $G$ the famous Lucchesi-Younger Theorem~\cite{LucY78} shows that the edge-disjoint version has Erd\H{o}s--P\'osa ratio $1$.
 For the vertex-disjoint version~\eqref{eq:lp} in planar graphs, Reed and Shepherd~\cite{ReeS96} gave the first constant upper bound on the Erd\H{o}s--P\'osa ratio.
 After three improvements by Fox and Pach as well as Cames van Batenburg, Esperet and M\"uller~\cite{CamEM17} and then Schlomberg, Thiele and Vygen~\cite{SchTV22}, this work decreases it below $8.38$.
 
 The next example of an uncrossable family is the set of all odd cycles in an undirected graph $G$.
 In this variant (in planar graphs) the edge-disjoint problem has an Erd\H{o}s--P\'osa ratio of exactly $2$~\cite{KraV04}.
 For the vertex-disjoint problem Fiorini et al.~\cite{FioHRV07} showed that the Erd\H{o}s--P\'osa ratio is at most $10$, which was improved to $6$ by Kr\'al\textquoteright{}, Sereni and Stacho~\cite{KraSS12}.
 
 Finally, one of the most interesting and well-studied variants of the Cycle Packing Problem is given as follows:
 Given a graph $G$ and a set $D$ of demand edges, then a $D$-cycle is a cycle in $G$ that contains exactly one demand edge.
 Since removing the demand edge from a $D$-cycle results in a path between the endpoints of the demand edge, the $D$-Cycle Packing Problem is equivalent to the Disjoint Paths Problem, and $D$-Cycle Packing in planar graphs corresponds to the Disjoint Paths Problem in fully planar instances.
 
 For the Fully Planar Edge-Disjoint Paths Problem the first constant-factor approximations and bounds on the integrality gap were given by Huang et al.~\cite{HuaMMSV21} and Garg, Kumar and Seb\H{o}~\cite{GarKS22}; the best upper bound on the integrality gap is $4$~\cite{GarKS22}.
 Garg and Kumar~\cite{GarK20} showed that also the Erd\H{o}s--P\'osa ratio for this problem is at most $4$.
 Due to a result by Middendorf and Pfeiffer~\cite{MidP93} the Fully Planar Vertex-Disjoint Paths Problem contains the edge-disjoint version as a special case.
 The first constant upper bound on the integrality gap of $5$ was found only recently by Schlomberg, Thiele and Vygen~\cite{SchTV22}.
 The best upper bound for the Erd\H{o}s--P\'osa ratio of $12$ comes in this variant from multiplying the upper bounds for the integrality gaps of the primal and the dual.
 
 This work now decreases both best-known upper bounds on the integrality gaps to the same value below $3.5$.
 For the Fully Planar Vertex-Disjoint Paths Problem we also improve the Erd\H{o}s--P\'osa ratio to below $8.38$:
 
 \begin{corollary}
  Given an instance $(G, D)$ of the Fully Planar Vertex-Disjoint Paths Problem, we can compute in polynomial time a set $\Pscr$ of vertex-disjoint $D$-cycles and $T\subseteq V(G)$ such that every $D$-cycle contains a vertex of $T$ with $|T| \leq 2.4 \cdot \frac{20 + \sqrt{130}}{9} |\Pscr| \leq 8.38 |\Pscr|$.
 \end{corollary}
 
 For most of the uncrossable families discussed above a lower bound of $2$ on the integrality gaps of~\eqref{eq:lp} and~\eqref{eq:lp_edgedisjoint} is known, which is also the best-known lower bound for general uncrossable families.
 Most of the corresponding examples can be constructed by modifying a $K_4$.
 Regarding the Erd\H{o}s--P\'osa ratio, the best-known lower bound for vertex-disjoint cycle packing (for uncrossable families) is still $2$, but for  edge-disjoint cycle packing and transversal Kr\'al\textquoteright{} (see~\cite{MaYZ13}) showed a lower bound of $4$ on the Erd\H{o}s--P\'osa ratio for the family of all cycles in $G$.
 See~\cite{SchTV22} for a more detailed overview on lower bounds for the integrality gaps and Erd\H{o}s--P\'osa ratios.
 
 There exist other examples of uncrossable cycle families that have been studied.
 For example, Rautenbach and Regen~\cite{RauR09} considered the Cycle Packing Problem with the family of shortest cycles in $G$, which is also uncrossable.
 Furthermore, the (uncrossable) family of all cycles that contain at least one vertex from a specified set $S \subseteq V(G)$ has been considered, for example by Goemans and Williamson~\cite{GoeW98}.
 This work yields the best-known upper bounds for the integrality gaps of the corresponding Cycle Packing LPs.
 
 For cycle families $\Cscr$ that are not uncrossable surprisingly few is known.
 For example, the set of all even cycles is not uncrossable.
 Here Göke et al.~\cite{GokKMS22} generalized Goemans and Williamson's~\cite{GoeW98} technique to get a constant upper bound on the vertex transversal LP;
 for the Cycle Packing Problem no constant-factor approximation algorithm is known.

 \newpage
 
\section{Preliminaries}\label{sec:preliminaries}

For the rest of the paper we fix a planar graph $G$, together with an embedding in the sphere $\mathbb{S}^2$, and an uncrossable family $\Cscr$ of cycles in $G$.
By a result by Schlomberg, Thiele and Vygen~\cite{SchTV22} we can compute optimum solutions to~\eqref{eq:lp} and~\eqref{eq:lp_edgedisjoint} with \emph{laminar} support, i.e.\@ any two cycles in the support can only ``touch'' but not ``cross'':

\begin{definition}\label{def:laminar_family}
 Let $G$ be a planar graph, embedded in the sphere.
 Deleting the embedding of a cycle $C$ in $G$ from the sphere results in two connected components of the sphere, which we call the \emph{sides} of $C$.
 Given a side $S$ of $C$ and another cycle $C^\prime$ in $G$, we say that $C^\prime$ is \emph{inside} $S$ or that $S$ \emph{contains} $C^\prime$ if $S$ contains a side of $C^\prime$.
 
 We call a family $\Lscr$ of cycles in $G$ \emph{laminar} if for any $C_1, C_2 \in \Lscr$ there exist sides $S_1$ of $C_1$ and $S_2$ of $C_2$ that are disjoint.
 \end{definition}
 
 \begin{definition}
 Let $\Lscr$ be a laminar family of cycles in a planar graph $G$, embedded in the sphere.
 By $V(\Lscr)$ and $E(\Lscr)$, respectively, we denote the set of all vertices, respectively edges, in cycles of $\Lscr$.
 
 The cycles corresponding to $\subseteq$-minimal sides in $\Lscr$ are called \emph{one-sided}, while the others are called \emph{two-sided}.
 For a one-sided cycle, the $\subseteq$-minimal side is also called \emph{one-sided}.
 
 We call two cycles $C_1, C_2 \in \Lscr$ \emph{homotopic} if there exist sides $S_1$ of $C_1$ and $S_2$ of $C_2$ that contain the same set of one-sided sides.
 
 For any connected component $D$ of $E(\Lscr)$ we call the set of all cycles in $\Lscr$ that are in $D$ a \emph{connected component} of $\Lscr$.
 
 A \emph{chain} is a laminar family of cycles with only two one-sided sides.
\end{definition}

It is easy to see that the sides $S_1$ and $S_2$ in Definition~\ref{def:laminar_family} are unique if $C_1 \neq C_2$.
Also, our notion of laminarity is equivalent to the definition in~\cite{SchTV22}.
In particular we can use the following Lemma from~\cite{SchTV22}:

\begin{lemma}\label{lemma:laminar_lp_solution}
 Let $G$ be a planar graph, embedded in the sphere, and $\Cscr$ an uncrossable family of cycles in $G$.
 Then there exist optimum solutions to the LPs~\eqref{eq:lp} and~\eqref{eq:lp_edgedisjoint} with laminar support.
 If $\Cscr$ has a weight oracle such solutions can be computed in polynomial time.
\end{lemma}

\section{Bounding the integrality gap}\label{sec:bounding_the_gap}

In this section we explain our main algorithm, which is a slight generalization of the greedy rounding algorithm used in~\cite{SchTV22}.
We first only analyze the easiest variant of the algorithm.
This already yields an upper bound of $4$ on the integrality gap for the cycle packing LP,
equalizing the best known upper bounds for edge-disjoint and vertex-disjoint Cycle Packing.
In Section~\ref{sec:improving_the_bounds} we will analyze a more refined version of the algorithm, which yields an upper bound of below $3.5$.

Here we only describe the algorithm for vertex-disjoint cycle packing.
The edge-disjoint version can be deduced similarly or sometimes even easier;
also there exists a reduction for laminar cycle families (\cite{SchTV22}) that allows us to immediately transfer our results from vertex-disjoint to edge-disjoint packing.
For more details we refer to Section~\ref{sec:edge_disjoint}.

\begin{definition}\label{def:neighbours}
  Let $\Lscr$ be a laminar family of cycles in a planar graph $G$, embedded in the sphere.
  Let $\Lscr_1$ be the set of one-sided cycles in $\Lscr$.
  For any $C \in \Lscr$ let $\Nscr_\Lscr(C)$ be the set of ``neighbours'' of $C$, i.e.\@ cycles in $\Lscr$ that contain a vertex of $C$.
  In particular, $C \in \Nscr_\Lscr(C)$.
  Define $\Nscr^1_\Lscr(C) := \Nscr_\Lscr(C) \cap \Lscr_1$ to be the set of one-sided ``neighbours'' of $C$.
  \end{definition}
 
 Given this definition, we can outline our algorithm: 
 
 \begin{algorithm}
 \caption{Greedy Rounding for Cycle Packing}\label{alg:main_alg}
 \hspace*{\algorithmicindent}
 \textbf{Input:} A planar graph $G$ and an uncrossable family $\Cscr$ of cycles in $G$.
 \newline
 \hspace*{\algorithmicindent}
 \textbf{Output:} A set $\Lscr^* \subseteq \Cscr$ of pairwise vertex-disjoint cycles.
  \begin{algorithmic}[1]
   \State\label{step:embed_G} Compute an embedding of $G$ in the sphere.
   \State\label{step:compute_lp_solution} Compute an optimum solution $x$ to the LP (\ref{eq:lp}) with
   laminar support.
   \While{$x \neq 0$}
   \State\label{step:make_x_structured} Modify $x$ to make it structured (see Definition~\ref{def:structured_lp_solution}).
   \State Let $\Lscr_x := \{C \in \Cscr : x_C > 0\}$ be the support of $x$.
   \State\label{step:choose_cycle_set} Pick a non-empty subset $\Fscr^* \subseteq \Lscr_x$ of pairwise vertex-disjoint cycles.
   \State Add all cycles in $\Fscr^*$ to the solution $\Lscr^*$.
   \State Set $x_C := 0$ for all $C \in \bigcup_{C^\prime \in \Fscr^*}\Nscr_{\Lscr_x}(C^\prime)$.
   \EndWhile
   \State Output $\Lscr^*$.
  \end{algorithmic}
 \end{algorithm}
 
 Throughout the algorithm we maintain a feasible solution $x$ to the LP~\eqref{eq:lp} with laminar support $\Lscr_x$ and a set $\Lscr^*$ of pairwise vertex-disjoint cycles with $V(\Lscr_x) \cap V(\Lscr^*) = \emptyset$.
 Step~\ref{step:embed_G} can be done in polynomial time~\cite{ChiNAO85}.
 For step~\ref{step:compute_lp_solution} we apply Lemma~\ref{lemma:laminar_lp_solution}.
 Then, in each iteration we add a set $\Fscr^*$ of cycles in the support of $x$ to $\Lscr^*$ and set $x$ on all neighbours of $\Fscr^*$ to $0$.
 
 We will make use of the following observation:
 If in each iteration we find a set $\Fscr^*$ with $x\left(\bigcup_{C \in \Fscr^*} \Nscr_{\Lscr_x}(C)\right) \leq \alpha |\Fscr^*|$ then the algorithm will return a solution of size at least $\frac{1}{\alpha}$ times the LP value.
 Thus, in order to bound the integrality gap of the cycle packing LP we only need to analyze the minimum values of $\frac{x\left(\bigcup_{C \in \Fscr^*} \Nscr_{\Lscr_x}(C)\right)}{|\Fscr^*|}$ that we can achieve with different choices of $\Fscr^*$.
 
 Note that after step~\ref{step:compute_lp_solution} the algorithm only operates on the (explicitly given) set of cycles in the support of the LP solution and does not depend on $\Cscr$ any more.
 Both the uncrossing property and the weight oracle are used only in this step.
 In particular, our results apply to any cycle family $\Cscr$ where step~\ref{step:compute_lp_solution} can be done, for example if $\Cscr$ is already laminar.
 
 We first explain in more detail what step~\ref{step:make_x_structured} does:
 
 \begin{definition}
 Let $\Lscr$ be a laminar family of cycles in a planar graph $G$, embeddded in the sphere.
 We call a two-sided cycle $C \in \Lscr$ \emph{redundant} if it is homotopic to a one-sided cycle in $\Lscr$ (cf.\@ Figure~\ref{fig:redundant_cycles}).
\end{definition}
 
 \begin{definition}\label{def:structured_lp_solution}
  Let $x \in \mathbb{R}^\Cscr$ be a solution to the LP~\eqref{eq:lp}.
  We call it \emph{structured} if the support of $x$ is laminar and each connected component $\Lscr$ of the support of $x$ contains no redundant cycles.
 \end{definition}

\begin{lemma}\label{lemma:no_redundant_cycles}
 Let $x \in \mathbb{R}^\Cscr$ be a feasible solution to the LP~\eqref{eq:lp} with laminar support $\Lscr$.
 Then we can compute a structured solution $x^\prime \in \mathbb{R}^\Lscr$ to~\eqref{eq:lp} with $\sum_{C \in \Lscr} x_C = \sum_{C \in \Lscr} x^\prime_C$ in polynomial time in the size of $\Lscr$.
\end{lemma}
\begin{proof}
 We can consider the connected components of $\Lscr$ separately, so we assume w.l.o.g.\@ that $E(\Lscr)$ is connected.
 Assume that $x$ is not structured.
 Let $C \in \Lscr$ be redundant with a side $S$ that contains no other redundant cycles.
 In particular, $S$ contains only one cycle $C^\prime \neq C$ in $\Lscr$.
 Since $E(\Lscr)$ is connected, $x(C) + x(C^\prime) \leq 1$ holds.
 Thus, we can shift the LP value from $C$ to $C^\prime$, i.e.\@ set $x^\prime(C^\prime) := x(C) + x(C^\prime)$ and $x^\prime(C) := 0$, removing $C$ from the support.
 This does not affect feasibility of the LP solution since it increases the LP value only on vertices strictly inside $S$, which are contained in no other cycles than $C^\prime$ due to minimality of $S$.
 See Figure~\ref{fig:redundant_cycles}.
 
 Applying this reduction at most $|\Lscr|$ times results in a solution as desired.
\end{proof}

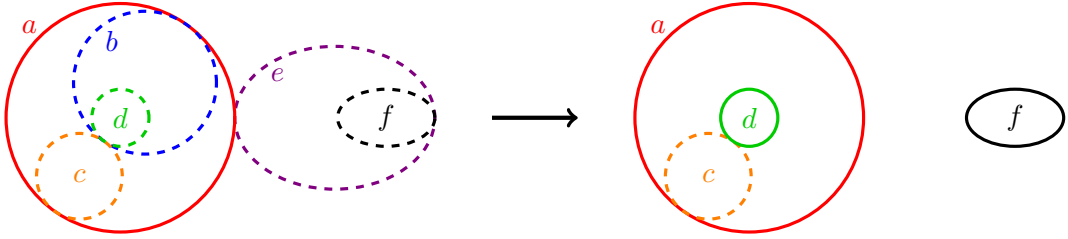
\begin{figure}[htb]
 \begin{center}
  \begin{tikzpicture}[scale=0.38, very thick]
   \draw[color=violet, dashed] (7.5,0) ellipse (3.5 and 2.5);
   \draw[color=red] (0,0) circle (4);
   \draw[dashed] (9.3,0) ellipse (1.7 and 1);
   \draw[color=blue, dashed] (55:1.5) circle (2.5);
   \draw[color=orange, dashed] (235:2.5) circle (1.5);
   \draw[color=darkgreen, dashed] (0,0) circle (1);
   \node[color=red] at (135:4.5) {$a$};
   \node[color=blue] at (-.3,2.7) {$b$};
   \node[color=orange] at (235:2.5) {$c$};
   \node[color=darkgreen] at (0,0) {$d$};
   \node[color=violet] at (5.5, 1.5) {$e$};
   \node[color=black] at (9.3,0) {$f$};
   \draw[ultra thick, ->] (13,0) -- (16,0);
   \begin{scope}[shift={(22,0)}]
   \draw[color=red] (0,0) circle (4);
   \draw[] (9.3,0) ellipse (1.7 and 1);
   \draw[color=orange, dashed] (235:2.5) circle (1.5);
   \draw[color=darkgreen] (0,0) circle (1);
   \node[color=red] at (135:4.5) {$a$};
   \node[color=orange] at (235:2.5) {$c$};
   \node[color=darkgreen] at (0,0) {$d$};
   \node[color=black] at (9.3,0) {$f$};
   \end{scope}
  \end{tikzpicture}
 \end{center}
 \caption{The left picture shows a possible laminar support of a feasible LP solution, consisting of six cycles.
 Dashed cycles have LP value $\frac{1}{3}$, while the others have LP value $\frac{2}{3}$.
 The cycles $b$ and $e$ are redundant because their interiors each contain only one other cycle; also $a$ is redundant because its exterior only contains the one-sided side of $f$.\newline
 In the proof of Lemma~\ref{lemma:no_redundant_cycles} we would pick the interiors of $b$ and $e$ as $S$ in the first and second step, increasing the LP value of $d$ and $f$ and removing $b$ and $e$ from the support.
 This yields a support as in the right image.
 The cycle $a$ is still redundant, however in the laminar family given by its connected component it is one-sided and therefore not redundant.
 Thus, the solution is structured.
 \label{fig:redundant_cycles}}
\end{figure}

  Next, we analyze the ratio $\frac{x\left(\bigcup_{C \in \Fscr^*} \Nscr_{\Lscr_x}(C)\right)}{|\Fscr^*|}$ that we can achieve.
  In this section we only consider the case that $\Fscr^*$ consists of a single one-sided cycle.
  We use the following Structure Lemma.
  The proof can be found in Section~\ref{sec:proof_of_main_lemma}.
  
  \begin{lemma}\label{lemma:structure_lemma}
  Let $\Lscr$ be a laminar family of cycles in a planar graph $G$, embedded in the sphere, such that $\Lscr$ contains no redundant cycles.
  Let $\Lscr_1$ be the set of one-sided cycles in $\Lscr$.
  Then there is a multi-subset $M^* \subseteq V(G)$ with $|M^*| \leq 3 |\Lscr_1|$ such that for any $C \in \Lscr$ we have $|M^* \cap V(C)| \geq |\Nscr_\Lscr^1(C) \setminus \{C\}|$.
 \end{lemma}
\begin{lemma}\label{lemma:x_is_small_on_average}
  Let $x$ be a structured solution to the vertex-disjoint cycle packing LP.
  Let $\Lscr$ be a connected component of the support of $x$ and $\Lscr_1$ the set of one-sided cycles in $\Lscr$.
  Then
  \begin{equation*} 
   \sum_{C \in \Lscr_1} x(\Nscr_\Lscr(C) \setminus \{C\}) \leq 3 |\Lscr_1|
  \end{equation*}
 \end{lemma}
 \begin{proof}
  By the Structure Lemma~\ref{lemma:structure_lemma}, choose $M^* \subseteq V(G)$ with $|M^*| \leq 3 |\Lscr_1|$ such that for any $C \in \Lscr$ we have $|M^* \cap V(C)| \geq |\Nscr_\Lscr^1(C) \setminus \{C\}|$.
  We get
  \begin{align*}
   & \sum_{C \in \Lscr_1} x(\Nscr_\Lscr(C) \setminus \{C\}) \\
   = & \ \sum_{C \in \Lscr} x(C) \cdot |\Nscr^1_\Lscr(C) \setminus \{C\}| \\
   \leq & \ \sum_{C \in \Lscr} x(C) \cdot |M^* \cap V(C)| \\
   \leq & \ \sum_{v \in M^*} \sum_{C \in \Lscr : v \in C} x(C) \\
   \leq & \ |M^*| \\
   \leq & \ 3 |\Lscr_1|
  \end{align*}
  \end{proof}
  
  Since the LP value of each single cycle itself is bounded by $1$ this immediately yields an upper bound of $4$ for the integrality gap of (\ref{eq:lp}):
  
  \begin{theorem}\label{thm:ub_4}
  Let $G$ be a planar graph, embedded in the sphere, and $\Cscr$ an uncrossable family of cycles in $G$.
  Then there exists an integral solution to the vertex-disjoint cycle packing LP with at least $\frac{1}{4}$ the LP value.
  If $\Cscr$ is given by a weight oracle we can compute such a solution in polynomial time.
 \end{theorem}
 \begin{proof}
  Let $x$ be an optimum solution to the LP~\eqref{eq:lp} with laminar support, as given by Lemma~\ref{lemma:laminar_lp_solution}.
  By applying Lemma~\ref{lemma:no_redundant_cycles} we can assume $x$ to be structured.
  Let $\Lscr_x := \{C \in \Cscr : x_C > 0\}$ be the (laminar) support of $x$.
  We proceed on each connected component of $\Lscr_x$ individually, so we may assume $E(\Lscr_x)$ to be connected.
  
  Let $\Lscr_1 \subseteq \Lscr_x$ be the set of one-sided cycles.
  In each step of our greedy rounding algorithm we add a one-sided cycle $C^*$ in $\Lscr_x$ to our solution and set $x$ on all cycles containing a vertex of $C^*$ to $0$, removing them from the support of $x$.
  
  Lemma~\ref{lemma:x_is_small_on_average} implies
  \begin{equation*}
   \sum_{C \in \Lscr_1} x(\Nscr_{\Lscr_x}(C))
   \leq 3 |\Lscr_1| + \sum_{C \in \Lscr_1} x(C)
   \leq 4 |\Lscr_1|
  \end{equation*}
  So there exists a one-sided cycle $C^*$ where removing $\Nscr_{\Lscr_x}(C^*)$ decreases $x$ by at most $4$.
  
 After the first iteration we again apply Lemma~\ref{lemma:no_redundant_cycles} and split the support of $x$ into connected components.
 Iterating this procedure until $x = 0$ yields a solution as desired.
 
 Note that Lemma~\ref{lemma:laminar_lp_solution} works in polynomial time if $\Cscr$ has a weight oracle.
 Thus, also the size of $\Lscr$ is polynomial in the size of $G$ and Lemma~\ref{lemma:no_redundant_cycles} also works in polynomial time.
 In each step we can find $C^*$ by picking the one-sided cycle in $\Lscr_x$ minimizing $x(\Nscr_{\Lscr_x}(C^*))$.
 \end{proof}
 
 The greedy rounding algorithm described above also allows for adding several cycles at once to the solution.
 We will exploit this in Section~\ref{sec:improving_the_bounds} to decrease the upper bound for the integrality gap to below $3.5$.
 
 \section{Proof of the Structure Lemma}\label{sec:proof_of_main_lemma}

 Before we prove the Structure Lemma~\ref{lemma:structure_lemma} for general uncrossable cycle families we first briefly consider the case where no two-sided cycles exist.
 In this case we start by constructing another planar graph $G^\prime$ on vertex set $\Lscr_1 = \Lscr$ as follows:
 
 For any vertex $v \in V(G)$ let $\Lscr_v \subseteq \Lscr$ be the set of cycles containing $v$.
 Since all cycles are one-sided, there is a natural cyclic order $\Lscr_v = \{C_1 =: C_{k+1}, C_2, \dots, C_k\}$ on $\Lscr_v$.
 Then we add for any $i=1, \dots, k$ the edge $\{C_i, C_{i+1}\}$ with its obvious planar embedding to $G^\prime$.
 Finally, we identify homotopic edges in $G^\prime$ (i.e.\@ parallel edges bounding an area homeomorphic to the disk).
 See Figure~\ref{fig:case_only_one_sided}.
 
 Now from $G^\prime$ we can construct our multi-set $M^*$:
 For each $e = \{C_1, C_2\} \in E(G^\prime)$ we add an arbitrary vertex in $V(C_1) \cap V(C_2)$ to $M^*$;
 furthermore, for each vertex $v \in V(G)$ that is contained in $k > 3$ cycles we add $k-3$ copies of $v$ to $M^*$.
 Since in this case $v$ lies inside a face of $G^\prime$ with exactly $k$ edges on its boundary we can construct another planar graph $G^*$ from $G^\prime$ by triangulating each such face $F$ with $k-3$ edges inside $F$ (cf.\@ Figure~\ref{fig:case_only_one_sided}).
 
 This yields a planar graph $G^*$ on vertex set $\Lscr_1$ with $|M^*|$ edges and no homotopic edges.
 Euler's formula implies $|M^*| = |E(G^*)| \leq 3 |V(G^*| - 6 = 3 |\Lscr_1| - 6$.
 
 Let now $C \in \Lscr$ and $\Bscr \subseteq \Nscr_\Lscr^1(C) \setminus \{C\}$ be the set of all neighbours of $C$ that are not connected to $C$ in $G^\prime$.
 By construction of $G^\prime$ this means that for any vertex $v \in V(C)$ that is contained in $k$ cycles at most $k-3$ of them can be in $\Bscr$.
 But we added $k-3$ copies of $v$ to $M^*$.
 This proves $|M^* \cap V(C)| \geq |\Bscr| + |\delta_{G^\prime}(C)| \geq |\Nscr_\Lscr^1(C) \setminus \{C\}|$.
 
 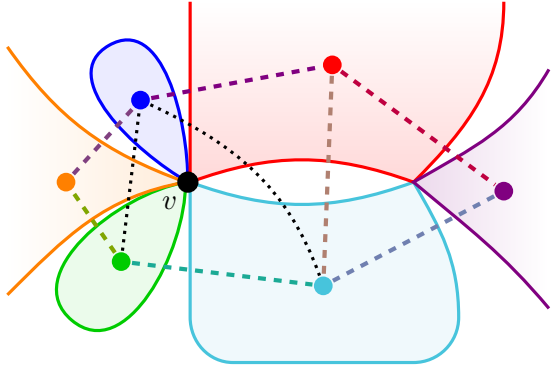
\begin{figure}[htb]
 \begin{center}
  \begin{tikzpicture}[scale=0.6, very thick]
   \fill[draw=white, top color=white, bottom color=red!15] (0.05,0) to[out=20, in=160] (5,0) to[out=45, in=-90] (7,4) -- (0.05, 4) -- (0.05,0);
   \fill[draw=white, left color=white, right color=orange!15] (-4,3) to[out=-55, in=160] (-0.05, 0) to[out=-170, in=45] (-4,-2.5) -- (-4,3);
   \fill[draw=white, left color=violet!15, right color=white] (8,2.5) to[out=-120,in=30] (5,0) to[out=-40,in=130] (8,-2.8) -- (8,2.5);
   
   \draw[blue, fill=blue!8] (0,0) to[out=90, in=30] (120:3.5) to[out=-150, in=150] (0, 0);
   \draw[darkgreen, fill=darkgreen!8] (-0.05,0) to[out=-170, in=140] (-130:4) to[out=-40, in=-90] (-0.05,0);
   \draw[SkyBlue, fill=SkyBlue!8] (0.05,0) to[out=-20, in=-160] (5,0) to[out=-60, in=90] (6,-3) to[out=-90,in=0] (5,-4) -- (1,-4) to[out=180,in=-90] (0.05,-3) -- (0.05,0);

   \draw[orange] (-4,3) to[out=-55, in=160] (-0.05, 0) to[out=-170, in=45] (-4,-2.5);
   \draw[red] (0.05,4) -- (0.05,0) to[out=20, in=160] (5,0) to[out=45, in=-90] (7,4);
   \draw[violet] (8,2.5) to[out=-120,in=30] (5,0) to[out=-40,in=130] (8,-2.8);
   
   \node at (-0.4, -0.45) {$v$};
   \draw[fill] (0,0) circle (0.2);
   \node[fill, circle, inner sep=2.5, color=darkgreen] (c1) at (-130:2.3) {};
   \node[fill, circle, inner sep=2.5, color=orange] (c2) at (-2.7,0) {};
   \node[fill, circle, inner sep=2.5, color=blue] (c3) at (120:2.1) {};
   \node[fill, circle, inner sep=2.5, color=red] (c4) at (3.2, 2.6) {};
   \node[fill, circle, inner sep=2.5, color=violet] (c5) at (7,-0.2) {};
   \node[fill, circle, inner sep=2.5, color=SkyBlue] (c6) at (3, -2.3) {};
   
   \draw[dashed, ultra thick, color=orange!50!darkgreen] (c1) to (c2);
   \draw[dashed, ultra thick, color=orange!50!blue] (c2) to (c3);
   \draw[dashed, ultra thick, color=red!50!blue] (c3) to (c4);
   \draw[dashed, ultra thick, color=red!50!violet] (c4) to (c5);
   \draw[dashed, ultra thick, color=violet!50!SkyBlue] (c5) to (c6);
   \draw[dashed, ultra thick, color=SkyBlue!50!darkgreen] (c1) to (c6);
   \draw[dashed, ultra thick, color=red!50!SkyBlue] (c4) to (c6);
   \draw[dotted] (c1) to (c3);
   \draw[dotted] (c3) to[out=-25, in=110] (c6);
  \end{tikzpicture}
 \end{center}
 \caption{Example for the case that no two-sided cycles exist:
 The coloured cycles are the elements of $\Lscr_1$.
 The vertices of $G^\prime$ are drawn as nodes inside the one-sided sides.
 The edges of $G^\prime$ are drawn as thick dashed lines.
 Since five cycles meet in the vertex $v$ we would add $v$ twice to $M^*$, in addition to the vertices in $M^*$ corresponding to edges of $G^\prime$.
 This is possible while keeping $|M^*| \leq 3 |\Lscr_1| - 6$ because we can triangulate the face of $G^\prime$ that $v$ lies in with two additional edges; as indicated by the dotted lines.
 \label{fig:case_only_one_sided}}
 \end{figure}
 
 Next, we have to consider also two-sided cycles.
 However, we do not know how to extend the relatively easy construction of $G^\prime$ and $G^*$ to this more general case.
 Instead, we will use the notion of \emph{incidences}:
 
 \begin{definition}
  Let $\Lscr$ be a laminar family of cycles in a planar graph $G$, embedded in the sphere.
  Let $\Lscr_1 \subseteq \Lscr$ be the set of one-sided cycles.
  A \emph{neighbour pair} is a pair $(\{C,N\}, v)$ of a set of two cycles $C, N \in \Lscr$ that are not homotopic and a vertex $v \in V(C) \cap V(N)$.
  We call two neighbour pairs $(\{C, N\}, v)$ and $(\{C, N\}, v^\prime)$ \emph{homotopic} if there exist $v$-$v^\prime$-paths $P$ in $C$ and $P^\prime$ in $N$ such that $P+P^\prime$ bounds an area that contains all one-sided sides in $\Lscr$.
  
  It is easy to see that homotopy defines an equivalence relation on neighbour pairs.
  An equivalence class of neighbour pairs for $C$ and $N$ is called an \emph{incidence} between $C$ and $N$ (cf.\@ Figure~\ref{fig:incidences}).
  The \emph{vertex set} $V(I)$ of an incidence $I$ between $C$ and $N$ is the set of all $v$ with $(\{C,N\}, v) \in I$.
  We also denote $I$ by $I = (\{C,N\}, V(I))$.
  For a cycle $C \in \Lscr$ let $\Iscr^1_\Lscr(C)$ be the set of all incidences between $C$ and one-sided cycles in $\Nscr_\Lscr^1(C)$.
  
  Let now $I$ be an incidence between $C \in \Lscr$ and $N \in \Nscr_\Lscr(C)$.
  Let $S_C$ be a side of $C$ and $S_N$ a side of $N$ such that $S_C$ and $S_N$ are disjoint.
  We call an incidence $I^\prime = (\{C^\prime, N^\prime\}, V(I^\prime))$ a \emph{sub-incidence} of $I$ if $C^\prime$ is inside $S_C$, $N^\prime$ is inside $S_N$ and $V(I^\prime) \subseteq V(I)$.
  We call $I$ \emph{minimal} if any sub-incidence of $I$ is equal to $I$.
  We call $I$ \emph{crossing} if $V(I) = \{v\}$ for some $v \in V(G)$ and there exist cycles $C_1, C_2 \in \Lscr$ that also contain $v$ with sides $S_1$ and $S_2$ such that $S_C, S_1, S_N, S_2$ are all disjoint and are ordered in this way around $v$.
  Such incidences are also called \emph{$v$-incidences}.
  If $I$ is not crossing we call it \emph{non-crossing} (cf.\@ Figure~\ref{fig:incidences}).
 \end{definition}
 
 Extending the idea of including the edges of $G^\prime$ in $M^*$, in order to prove Lemma~\ref{lemma:structure_lemma} we will construct a set of incidences instead of a set of vertices.
 
 \begin{definition}\label{def:structured_incidence_set}
  Let $\Lscr$ be a laminar family of cycles in a planar graph $G$, embedded in the sphere.
  Let $\Lscr_1 \subseteq \Lscr$ be the set of one-sided cycles.
  Let $M$ be a multi-set of incidences in $\Lscr$.
  We say that an element $I \in M$ \emph{hits} a cycle $C \in \Lscr$ if $V(I) \subseteq V(C)$.
  We call a cycle $C \in \Lscr$ \emph{$M$-good} if at least $|\Iscr_\Lscr^1(C)|$ elements of $M$ hit $C$.
  We call $M$ \emph{good} if all cycles in $\Lscr$ are $M$-good and $|M| \leq 3|\Lscr_1| - 6$.
  Furthermore, we call $M$ \emph{structured} if the following properties hold:
  \begin{enumerate}
   \item $M$ contains every non-crossing incidence between one-sided cycles.
   \item For each $C \in \Lscr_1$ and $v \in V(C)$ there exist at least as many $v$-incidences in $M$ as there are $v$-incidences between $C$ and $\Nscr_\Lscr^1(C)$ in $\Lscr$.
  \end{enumerate}
 \end{definition}
 
 This notion of structured incidence sets is inspired from the construction of $M^*$ in the case $\Lscr = \Lscr_1$:
 The edges in $G^\prime$ correspond to non-crossing incidences between one-sided cycles, which are included in $M$ by property 1.
 Property 2 makes sure that vertices in which many one-sided cycles meet are included in $M$.
 In particular, we get the following as a direct consequence of the above definition:
 
 \begin{lemma}\label{lemma:structured_implies_good}
  Let $\Lscr$ be a laminar family of cycles in a planar graph $G$, embedded in the sphere.
  Let $M$ be a structured set of incidences in $\Lscr$.
  Then every one-sided cycle is $M$-good.
 \end{lemma}
 
 \begin{figure}[htb]
 \begin{center}
  \begin{tikzpicture}[scale=0.6, very thick]
   \fill[draw=white, top color=blue!15, bottom color=white] (-2,-4) to[out=85,in=230] (0,0) to[out=-15,in=195] (4.95,0) -- (4.95,-4) -- (-2, -4);
   
   \fill[draw=white, top color=violet!15, bottom color=white] (10,-4) to[out=105,in=0] (5.05,0) -- (5.05,-4) -- (10,-4);
   
   \draw[SkyBlue, fill=SkyBlue!8] (0, 0) to[out=15, in=165] (5,0) to[out=110, in=70] (0, 0);
   
   \draw[darkgreen, fill=darkgreen!8] (0,0) to[out=120, in=60] (150:3.5) to[out=240, in=180] (0,0);
   
   \begin{scope}[shift={(5,0)}]
   \draw[red, fill=red!8] (0,0) to[out=30,in=-30] (60:3.5) to[out=150, in=90] (0,0);
   \node[thick, red] (c3) at (60:2) {$C_3$};
   \end{scope}
   
   \draw[orange, dashed] (-2.5,4) to[out=-50,in=105] (0,0) to[out=80, in=180] (2.5, 1.9) to[out=0,in=100] (5,0) to[out=15, in=-100] (8.5,4);
   \draw[blue] (-2,-4) to[out=85,in=230] (0,0) to[out=-15,in=195] (4.95,0) -- (4.95,-4);
   \draw[violet] (10,-4) to[out=105,in=0] (5.05,0) -- (5.05,-4);
   \draw[very thick, dashed, ->] (11,0) -- (8,0);
   \draw[very thick, dashed, ->] (8,0) -- (2.5,0);
   \draw[very thick, dashed, ->] (2.5,0) -- (-3,0);

   \node at (0.15, -0.3) {$v$};
   \node at (5.4, -0.3) {$w$};
   \node[thick, darkgreen] (c1) at (150:2) {$C_1$};
   \node[thick, SkyBlue] (c2) at (2.5,0.9) {$C_2$};
   \node[thick, blue] (c4) at (2.5,-2) {$C_4$};
   \node[thick, violet] (c5) at (7,-2) {$C_5$};
   \node[thick, orange] (c6) at (-1.3,3.8) {$C_6$};
   \node[thick] at (-2.6, -0.5) {$C^*$};
   \node[thick] at (10.5,-1) {$S_1$};
   \node[thick] at (10.5,1) {$S_2$};
   \draw[thick, violet, ->, dotted] (c5) to[bend right] (c3);
   \draw[thick, red, ->, dotted] (c3) to[] (c5);
   \draw[thick, blue, ->, dotted] (c4) to[bend left] (c3);
   \draw[thick, SkyBlue, ->, dotted] (c2) to[out=200, in=120] (c4);
   \draw[thick, darkgreen, ->, dotted] (c1) to[bend right] (c4);
  \end{tikzpicture}
 \end{center}
 \caption{An example of a part of the embedding of $\Lscr$.
 One-sided cycles are filled, two-sided cycles are drawn dashed.
 The cycle $C^*$ has exactly one non-crossing incidence to each of the five one-sided cycles; the neighbour pairs $(\{C^*, C_6\}, v)$ and $(\{C^*, C_6\}, w)$ are not homotopic and therefore yield two incidences:
 $(\{C^*, C_6\}, \{v\})$ is a $v$-incidence, while $(\{C^*, C_6\}, \{w\})$ is non-crossing.
 \newline
 The lower half of the picture corresponds to the side $S_1$ which contains no two-sided cycles;
 the arrows on $C^*$ indicate the orientation of $C^*_\rightarrow$.
 For each one-sided cycle $C_i$, $1 \leq i \leq 5$, the dotted arrow starting at $C_i$ represents the element $f_j(I) \in M_j^\prime$ that replaces the non-crossing incidence $I$ between $C_i$ and $C^*$ in $M_j$, where $S_j$ is the side containing $C_i$.
 For example, $f_1((\{C^*, C_4\}, \{v,w\})) = (\{C_3, C_4\}, \{w\})$ because $C_3$ is the ``first'' cycle in $S_2$ (w.r.t.\@ $C^*_\rightarrow$) touching $C_4$, while $f_2((\{C^*, C_2\}, \{v,w\})) = (\{C_4, C_2\}, \{v,w\})$ because $(\{C_5, C_2\}, \{w\})$ is crossing.
 \newline
 For proving that $C_6$ is $M$-good we need to consider consecutive incidences in $A$ w.r.t.\@ the ordering inherited by $C^*_\rightarrow$ with the same image under $g$, like $a_1=(\{C_6, C_5\},\{w\})$ and $a_2=(\{C_6, C_4\}, \{w\})$.
 However in this case $f_1((\{C^*,C_4\},\{v,w\}))$ must hit $C_6$.
 \label{fig:incidences}}
 \end{figure}
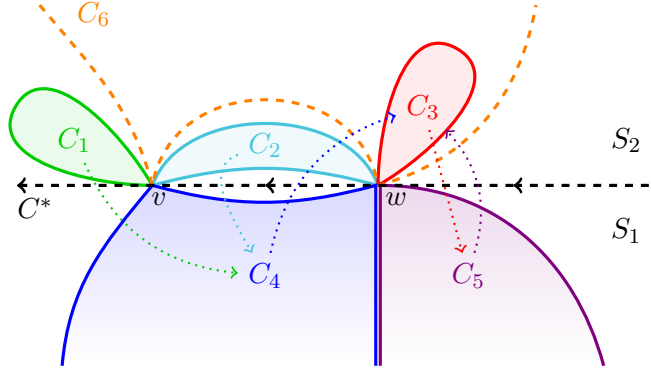
 
 \begin{lemma}\label{lemma:structure_lemma_using_incidences}
  Let $\Lscr$ be a laminar family of at least two cycles in a planar graph $G$, embedded in the sphere.
  Let $\Lscr_1$ be the set of one-sided cycles in $\Lscr$.
  Then there exists a good and structured set $M$ of incidences in $\Lscr$.
 \end{lemma}
 \begin{proof}
  We use an induction on $|\Lscr|$.
  We start with the case $|\Lscr_1| = 2$, which is trivial because all cycles are homotopic and there are no incidences.
  In particular, $M := \emptyset$ yields a good and structured set.
  
  Let us now assume $|\Lscr| = |\Lscr_1| = 3$.
  In this case there is at most one incidence between any two cycles, thus let $M$ be the set of all incidences, which is a good and structured set.
  
  Let now $|\Lscr| > 3$.
  We first assume $\Lscr \neq \Lscr_1$; we will show how to deal with the other case at the end.
  Let $C^* \in \Lscr$ be two-sided with a side $S_1$ that contains no other two-sided cycles.
  Let $S_2$ be the other side of $C^*$.
  For $i=1,2$ let $\Lscr_{S_i} \subseteq \Lscr$ be the set of all cycles with a side inside $S_i$.
  In particular, $C^* \in \Lscr_{S_1} \cap \Lscr_{S_2}$ is one-sided in both families.
  Since both families are strictly smaller than $\Lscr$, by induction hypothesis we can find good and structured sets $M_i$ for $\Lscr_{S_i}$ for $i=1,2$.
  
  Let $C^*_\rightarrow$ be an arbitrary (Eulerian) orientation of $C^*$.
  Assume $I \in \Iscr^1_{\Lscr_{S_i}}(C^*)$ to be a non-crossing incidence between $C^*$ and a cycle $N \in \Lscr_1 \cap \Lscr_{S_i}$ for some $i \in \{1, 2\}$.
  Let $\Iscr_I$ be the set of minimal sub-incidences of $I$ in $\Lscr$ between $N$ and cycles in $\Lscr_{S_{3-i}}$.
  The orientation of $C^*_\rightarrow$ induces a natural linear order on $\Iscr_I$.
  Let $f_1(I)$ be a first element in this order.
  Let $f_2(I)$ be a first non-crossing incidence in this order if it exists and $f_2(I) = f_1(I)$ otherwise (cf.\@ Figure~\ref{fig:incidences}).
  Let $M_i^\prime$ arise from $M_i$ by replacing each non-crossing incidence $I \in \Iscr_{\Lscr_{S_i}}(C^*)$ by $f_i(I)$.
  Since $V(f_i(I)) \subseteq V(I)$ holds for any $I$, $M_i^\prime$ is also good for $\Lscr_{S_i}$.

  Let $M := M_1^\prime + M_2^\prime$.
  We first show that $M$ is structured.
  To show property 1 of Definition~\ref{def:structured_incidence_set}, let $C_1, C_2 \in \Lscr_1$ and $I = (\{C_1, C_2\}, V(I))$ be a non-crossing incidence.
  If $C_1$ and $C_2$ are on the same side $S_i$ of $C^*$ then $I \in M_i \setminus \Iscr^1_{\Lscr_{S_i}}(C^*) \subseteq M$.
  Otherwise, w.l.o.g.\@ $C_i$ is inside $S_i$ for $i=1,2$.
  First assume that there is an $i \in \{1,2\}$ such that $C_i$ is the only one-sided cycle in $S_i$.
  Then w.l.o.g.\@ $S_{3-i}$ contains two one-sided cycles.
  Thus, there is a non-crossing sub-incidence $I_{3-i} \in \Iscr_{\Lscr_{S_{3-i}}}^1(C^*)$ of $I$.
  Also, $|\Iscr_{I_{3-i}}| = 1$ and therefore $f_{3-i}(I_{3-i}) = I \in M$.
  Now we assume that both $S_i$ contain at least two one-sided cycles.
  Again, there exist non-crossing sub-incidences $I_i \in \Iscr_{\Lscr_{S_i}}^1(C^*)$ of $I$ for each $i = 1,2$.
  Since $M_i$ is structured, $I_i \in M_i$.
  Let $C_1^\prime \in \Lscr_{S_1}$ and $C_2^\prime \in \Lscr_{S_2}$ such that $f_i(I_i)$ is an incidence between $C_i$ and $C_{3-i}^\prime$ for $i=1,2$.
  If neither $C_1 = C_1^\prime$ nor $C_2 = C_2^\prime$ holds then all four of those cycles must meet in one vertex due to minimality of $f_i(I_i)$ in the order on $\Iscr_{I_i}$.
  But then both $f_1(I_1)$ and $f_2(I_2)$ are crossing, contradicting the definition of $f_2$.
  Thus, we get $f_1(I_1) = I$ or $f_2(I_2) = I$.
  
  Next, we show property 2 of Definition~\ref{def:structured_incidence_set}.
  Let $C \in \Lscr_1 \cap \Lscr_{S_i}$ for some $i \in \{1,2\}$ and $v \in V(C)$.
  If all one-sided cycles that contain $v$ are in the same $\Lscr_{S_i}$ then $M_i$ already contains enough $v$-incidences.
  So we only have to consider the case $v \in V(C^*)$ such that both $\Lscr_{S_1} \cap \Lscr_1$ and $\Lscr_{S_2} \cap \Lscr_1$ contain a cycle that contains $v$.
  Let $\Nscr \subseteq \Lscr_1$ be the set of one-sided cycles with a $v$-incidence to $C$.
  
  Case 1: There is a $j \in \{1,2\}$ such that the cycles in $\Lscr_{S_j}$ that contain $v$ form a chain.
  In this case there is a natural correspondence between $v$-incidences between one-sided cycles in $\Lscr_{S_{3-j}}$ and $v$-incidences between one-sided cycles in $\Lscr$ (cf.\@ Figure~\ref{fig:m_is_structured}).
  Thus, $M_{3-j}$ already contains $|\Nscr|$ $v$-incidences.
  
  Case 2: For $j=1,2$ the cycles in $\Lscr_{S_j}$ that contain $v$ do not build a chain.
  Since $S_1$ contains no two-sided cycles except $C^*$, this implies that by definition of $f_1$ there is a non-crossing incidence $I$ in $M_1$ such that $f_1(I)$ is a $v$-incidence, as shown in Figure~\ref{fig:m_is_structured}.
  Also, $M_i$ contains at least $|\Nscr \cap \Lscr_{S_i}|$ $v$-incidences and $M_{3-i}$ contains at least $|\Nscr \cap \Lscr_{S_{3-i}}| - 2$ $v$-incidences because of property 2 of Definition~\ref{def:structured_incidence_set} for $C^*$ and $v$ in $\Lscr_{S_{3-i}}$.
  However, if the second lowerbound is tight then also the incidence between $C^*$ and $C$ in $\Lscr_{S_i}$ is crossing and $M_1$ contains even $|\Nscr \cap \Lscr_{S_i}| + 1$ $v$-incidences.
  So in total $M$ contains at least $|\Nscr|$ $v$-incidences, concluding the proof that $M$ is structured.
  
  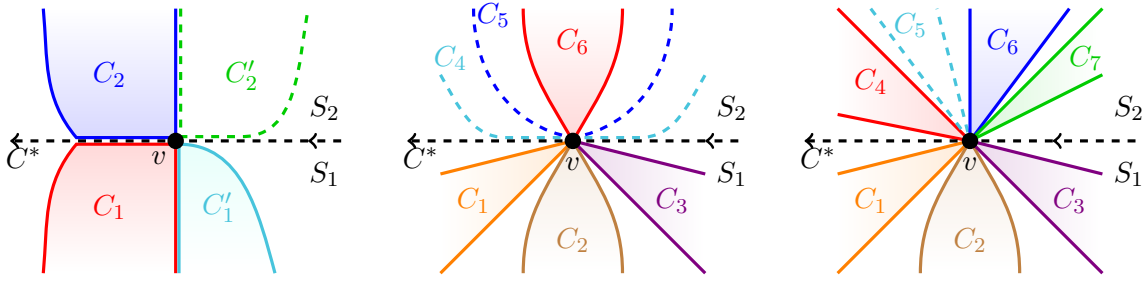
\begin{figure}[htb]
 \begin{center}
  \begin{tikzpicture}[scale=0.44, very thick]
   \fill[draw=white, top color=red!15, bottom color=white] (-4,-4) to[out=85,in=230] (-3,-0.1) -- (0,-0.1) -- (0,-4) -- (-4, -4);
   
   \fill[draw=white, top color=white, bottom color=blue!15] (-4,4) to[out=-85,in=-230] (-3,0.1) -- (0,0.1) -- (0,4) -- (-4,4);
   
   \fill[draw=white, top color=SkyBlue!15, bottom color=white] (3,-4) to[out=105,in=0] (0.1,-0.1) -- (0.1,-4) -- (3,-4);
   
   \draw[darkgreen, dashed] (0.15,4) -- (0.15,0.13) -- (2,0.13) to[out=0, in=-100] (4,4);
   \draw[red] (-4,-4) to[out=85,in=230] (-3,-0.1) -- (0,-0.1) -- (0,-4);
   \draw[blue] (-4,4) to[out=-85,in=-230] (-3,0.1) -- (0,0.1) -- (0,4);
   \draw[SkyBlue] (3,-4) to[out=105,in=0] (0.1,-0.1) -- (0.1,-4);
   \draw[very thick, dashed, ->] (5,0) -- (4,0);
   \draw[very thick, dashed, ->] (4,0) -- (-5,0);

   \node at (-0.5, -0.5) {$v$};
   \draw[black, fill] (0,0) circle (0.2);
   \node[thick, darkgreen] at (2,2) {$C_2^\prime$};
   \node[thick, red] at (-2,-2) {$C_1$};
   \node[thick, SkyBlue] at (1.3,-2) {$C_1^\prime$};
   \node[thick, blue] at (-2,2) {$C_2$};
   \node[thick] at (-4.6, -0.5) {$C^*$};
   \node[thick] at (4.5,-1) {$S_1$};
   \node[thick] at (4.5,1) {$S_2$};
   
   \begin{scope}[shift={(12,0)}]
   \fill[draw=white, top color=white, bottom color=red!15] (-1.5,4) to[out=-90,in=120] (0,0) to[out=60,in=-90] (1.5,4) -- (-1.5,4);
   \fill[draw=white, left color=white, right color=orange!15] (-4,-1) -- (0,0) -- (-4,-4) -- (-4,-1);
   \fill[draw=white, top color=brown!15, bottom color=white] (-1.5,-4) to[out=90,in=-120] (0,0) to[out=-60,in=90] (1.5,-4) -- (-1.5,-4);
   \fill[draw=white, left color=violet!15, right color=white] (4,-1) -- (0,0) -- (4,-4) -- (4,-1);
   
   \draw[dashed, color=SkyBlue] (-4,2) to[out=-60,in=180] (-2,0.1) -- (2,0.1) to[out=0,in=-120] (4,2);
   \draw[dashed, color=blue] (-3,4) to[out=-90,in=170] (0,0.1) to[out=10,in=-90] (3,4);
   \draw[color=red] (-1.5,4) to[out=-90,in=120] (0,0) to[out=60,in=-90] (1.5,4);
   \draw[color=orange] (-4,-1) -- (0,0) -- (-4,-4);
   \draw[color=brown] (-1.5,-4) to[out=90,in=-120] (0,0) to[out=-60,in=90] (1.5,-4);
   \draw[color=violet] (4,-1) -- (0,0) -- (4,-4);
   
    \draw[very thick, dashed, ->] (5,0) -- (4, 0);
    \draw[very thick, dashed, ->] (4,0) -- (-5,0);
    \node[thick] at (-4.6, -0.5) {$C^*$};
   \node[thick] at (4.8,-1) {$S_1$};
   \node[thick] at (4.8,1) {$S_2$};
   \node at (0, -0.7) {$v$};
   \draw[black, fill] (0,0) circle (0.2);
   \node[red] at (0,3) {$C_6$};
   \node[SkyBlue] at (-3.7, 2.5) {$C_4$};
   \node[blue] at (-2.4,3.8) {$C_5$};
   \node[orange] at (-3,-1.8) {$C_1$};
   \node[brown] at (0,-3) {$C_2$};
   \node[violet] at (3,-1.8) {$C_3$};
   \end{scope}
   
   \begin{scope}[shift={(24,0)}]
   \fill[draw=white, right color=red!15, left color=white] (-4,0.8) -- (0,0) -- (-4,4) -- (-4,0.8);
   \fill[draw=white, top color=white, bottom color=blue!15] (0,4) -- (0,0) -- (3,4) -- (0,4);
   \fill[draw=white, left color=darkgreen!15, right color=white] (4,4) -- (0,0) -- (4,2) -- (4,4);
   \fill[draw=white, left color=white, right color=orange!15] (-4,-1) -- (0,0) -- (-4,-4) -- (-4,-1);
   \fill[draw=white, top color=brown!15, bottom color=white] (-1.5,-4) to[out=90,in=-120] (0,0) to[out=-60,in=90] (1.5,-4) -- (-1.5,-4);
   \fill[draw=white, left color=violet!15, right color=white] (4,-1) -- (0,0) -- (4,-4) -- (4,-1);
   
   \draw[color=red] (-4,0.8) -- (0,0) -- (-4,4);
   \draw[dashed, color=SkyBlue] (-3,4) --(0,0) -- (-1,4);
   \draw[color=blue] (0,4) -- (0,0) -- (3,4);
   \draw[color=darkgreen] (4,4) -- (0,0) -- (4,2);
   \draw[color=orange] (-4,-1) -- (0,0) -- (-4,-4);
   \draw[color=brown] (-1.5,-4) to[out=90,in=-120] (0,0) to[out=-60,in=90] (1.5,-4);
   \draw[color=violet] (4,-1) -- (0,0) -- (4,-4);
   
    \draw[very thick, dashed, ->] (5,0) -- (2.5, 0);
    \draw[very thick, dashed, ->] (2.5,0) -- (-5,0);
    \node[thick] at (-4.6, -0.5) {$C^*$};
   \node[thick] at (4.8,-1) {$S_1$};
   \node[thick] at (4.8,1) {$S_2$};
   \node at (0, -0.7) {$v$};
   \draw[black, fill] (0,0) circle (0.2);
   \node[red] at (-3, 1.8) {$C_4$};
   \node[SkyBlue] at (-1.8, 3.5) {$C_5$};
   \node[blue] at (1,3) {$C_6$};
   \node[darkgreen] at (3.5,2.5) {$C_7$};
   \node[orange] at (-3,-1.8) {$C_1$};
   \node[brown] at (0,-3) {$C_2$};
   \node[violet] at (3,-1.8) {$C_3$};
   \end{scope}
  \end{tikzpicture}
 \end{center}
 \caption{Three examples for the proof that $M$ is structured:
 The left image shows a situation with a non-crossing incidence $I$ between $C_1$ and $C_2$.
 If for each $i=1,2$ the incidence between $C_i$ and $C^*$ is mapped by $f_i$ to an incidence between $C_i$ and $C_{3-i}^\prime \neq C_{3-i}$ then each $C_i^\prime$ must come ``before'' $C_i$ w.r.t.\@ $C^*_\rightarrow$.
 So all cycles meet in a point $v$, contradicting the choice of $f_2$.
 \newline
 The other two images show the different situations for property 2:
 The middle image shows case 1 where all cycles on one side (in this case $S_2$) form a chain.
 Then it is easy to see that $M_1$ already contains enough $v$-incidences.
 \newline
 The right image shows the remaining case 2.
 Let us assume that $C$ is inside $S_2$.
 $M_2$ already contains $|\Nscr \cap \Lscr_{S_2}|$ $v$-incidences and $M_1$ contains at least $|\Nscr \cap \Lscr_{S_1}| - 2$ $v$-incidences because $C_1$ and $C_3$ are the only cycles that might be in $\Nscr$ without a $v$-incidence to $C^*$.
 However, if both $C_1$ and $C_3$ are in $\Nscr$ then $M_2$ contains an additional $v$-incidence because $(\{C^*, C\}, \{v\})$ must be crossing.
 \label{fig:m_is_structured}}
 \end{figure}
  
  It is left to show that $M$ is good.
  First, we have
  \begin{equation*}
   |M| \leq 3 (|\Lscr_1 \cap \Lscr_{S_1}| + 1) - 6 + 3 (|\Lscr_1 \cap \Lscr_{S_2}| + 1) - 6 \leq 3 |\Lscr_1| - 6
  \end{equation*}
  Also all one-sided cycles are $M$-good since $M$ is structured.
  Let now $C \in \Lscr$ be two-sided.
  By choice of $C^*$, $C \in \Lscr_{S_2}$.
  Let $A$ be the set of incidences between $C$ and one-sided cycles inside $S_1$ in $\Lscr$, and let $B$ be the set of incidences between $C$ and $C^*$ in the family $\Lscr_{S_2}$.
  Note that $B=\emptyset$ if $C=C^*$ or $C$ is homotopic to $C^*$.
  By definition we get $|\Iscr_\Lscr(C)| - |\Iscr_{\Lscr_{S_2}}(C)| = |A| - |B|$.
  Also, $A$ is naturally ordered cyclically by $C^*_\rightarrow$.
  
  We first assume $B \neq \emptyset$.
  Then there is a natural map $g \colon A \to B$ such that each $a \in A$ is a sub-incidence of $g(a)$;
  furthermore, for each $b \in B$ the cyclic order on $A$ induces a linear order on $g^{-1}(\{b\})$.
  Let now $b \in B$ and $a_1, a_2 \in A$ be consecutive elements in the linear order on $g^{-1}(\{b\})$.
  Then there is an incidence $I_2 \in \Iscr_{\Lscr_{S_1}}(C^*)$ such that $a_2$ is a sub-incidence of $I_2$.
  Then by the definition of $f_1$ we have that $f_1(I_2)$ hits $C$ (cf.\@ Figure~\ref{fig:incidences}).
  
  In the other case $B = \emptyset$ one can see similarly that for any consecutive elements $a_1, a_2 \in A$ in the cyclic order on $A$ there is an incidence $I_2 \in \Iscr_{\Lscr_{S_1}}(C^*)$ that $a_2$ is a sub-incidence of such that $f_1(I_2)$ hits $C$.
  Thus, in both cases the number of incidences in $M_1^\prime$ that hit $C$ is at least $|A| - |B|$ and
  \begin{align*}
   |\{I \in M : V(I) \subseteq V(C)\}| \geq & \ |\{I \in M_1^\prime : V(I) \subseteq V(C)\}| + |\{I \in M_2^\prime : V(I) \subseteq V(C)\}| \\
   \geq & \ |\{I \in M_1^\prime : V(I) \subseteq V(C)\}| + |\Iscr_{\Lscr_{S_2}}(C)| \\
   = & \ |\{I \in M_1^\prime : V(I) \subseteq V(C)\}| + |\Iscr_\Lscr(C)| - (|A| - |B|) \\
   \geq & \ |\Iscr_\Lscr(C)|
  \end{align*}
  
  We are left with the case that $|\Lscr| > 3$, but $\Lscr = \Lscr_1$.
  In this case, however, we can just add a new two-sided cycle $C^*$ with at least two one-sided sides on each side both to $G$ and to $\Lscr$.
  This only strengthens the statement we prove because any good and structured set for the constructed instance is also good and structured for the original instance.
  Also, in the induction step we can still apply the induction hypothesis to $\Lscr_{S_1}$ and $\Lscr_{S_2}$ because both $S_1$ and $S_2$ contain at least two one-sided cycles.
 \end{proof}
 
 As a direct consequence of this lemma we get the Structure Lemma~\ref{lemma:structure_lemma}:
 
 \begin{replemma}{lemma:structure_lemma}
  Let $\Lscr$ be a laminar family of cycles in a planar graph $G$, embedded in the sphere, such that $\Lscr$ contains no redundant cycles.
  Let $\Lscr_1$ be the set of one-sided cycles in $\Lscr$.
  Then there is a multi-subset $M^* \subseteq V(G)$ with $|M^*| \leq 3 |\Lscr_1|$ such that for any $C \in \Lscr$ we have $|M^* \cap V(C)| \geq |\Nscr_\Lscr^1(C) \setminus \{C\}|$.
 \end{replemma}
 \begin{proof}
  W.l.o.g.\@ $|\Lscr| \geq 2$.
  By Lemma~\ref{lemma:structure_lemma_using_incidences}
  let $M$ be a good set of incidences in $\Lscr$.
  Let $M^*$ arise from adding one element of $V(I)$ for every $I \in M$.
  In particular, $|M^*| = |M| \leq 3|\Lscr_1| - 6$.
  
  Let $C \in \Lscr$.
  Since no side of $C$ is redundant, it is not homotopic to any cycle in $\Lscr_1$.
  Thus,
  \begin{equation*} 
   |\Nscr_\Lscr^1(C) \setminus \{C\}| \leq |\Iscr_\Lscr^1(C)| \leq |\{I \in M : V(I) \subseteq V(C)\}| \leq |M^* \cap V(C)|
  \end{equation*}
 \end{proof}
 
 \section{Improving the bounds below 3.5}\label{sec:improving_the_bounds}
 
 In this section we improve the bound on the integrality gaps of the LPs (\ref{eq:lp}) and (\ref{eq:lp_edgedisjoint}) to $\frac{20 + \sqrt{130}}{9} < 3.5$.
 We first only consider the vertex-disjoint cycle packing LP (\ref{eq:lp});
 an extension to the edge-disjoint case is given in Section~\ref{sec:edge_disjoint}.
 
 \subsection{The vertex-disjoint version}
 We still use Algorithm~\ref{alg:main_alg} from Section~\ref{sec:bounding_the_gap} for the improved approximation guarantee, but add further possibilities for the set $\Fscr^*$ of cycles that are added to our solution during a single iteration.
 Note that the one-cardinality set chosen in Theorem~\ref{thm:ub_4} already gives a good approximation guarantee if the average LP value on one-sided cycles is small.
 On the other hand, if the average LP value on one-sided cycles is large then we will find a large set of pairwise vertex-disjoint cycles with relatively small neighbourhood which we can take as $\Fscr^*$.
 For analyzing the case of $\Fscr^*$ containing more than one cycle we use the following Lemma:
 
 \begin{lemma}\label{lemma:neighbourhood_of_large_set}
  Let $\Lscr$ be a laminar family of cycles in a planar graph $G$, embedded in the sphere.
  Let $\Lscr_1$ be the set of one-sided cycles in $\Lscr$.
  Let $\Fscr \subseteq \Lscr_1$.
  Then there is a set $M \subseteq V(\Fscr)$ with $|M| \leq |\Fscr| + |\Lscr_1|$ such that each cycle in $\Lscr$ is either vertex-disjoint to all cycles in $\Fscr$ or contains a vertex from $M$.
 \end{lemma}
 \begin{proof}
  W.l.o.g.\@ $|\Lscr| > 1$.
  We can also assume that there exists a point $p_\infty$ on the sphere that lies neither on the embedding of vertices or edges in $G$ nor in any one-sided side of a cycle in $\Lscr_1$ (if this is not the case we can replace an arbitrary edge in $G$ by two parallel edges, which does not affect the lemma's statement).
  For this proof, we call the side of a cycle $C \in \Lscr$ that does not contain $p_\infty$ the interior of $C$ and say $C$ contains a cycle $C^\prime \in \Lscr$, or $C^\prime \inside C$, if the interior of $C^\prime$ is contained in the interior of $C$.
  
  Let $\Bscr_\text{int} \subseteq \Lscr$ be the set of cycles $C$ such that there is a cycle $C^\prime \in \Fscr$ with $C^\prime \subseteq C$ and $V(C) \cap V(C^\prime) \neq \emptyset$.
  In particular, $\Fscr \subseteq \Bscr_\text{int}$.
  Let $f \colon \Bscr_\text{int} \to \Fscr$ such that each $C \in \Bscr_\text{int}$ contains $f(C)$ and shares a vertex with $f(C)$.
  Then for any $C \in \Fscr$ all cycles in $f^{-1}(C)$ must build a chain and therefore meet in some vertex $v_C \in V(C)$.
  Thus, $M_\text{int} := \{v_C : C \in \Fscr\}$ hits all cycles in $\Bscr_\text{int}$.
  
  Let now $\Bscr_\text{ext} \subseteq \Lscr \setminus \Bscr_\text{int}$ be the set of all cycles in $\Lscr \setminus \Bscr_\text{int}$ that share a vertex with any cycle in $\Fscr$.
  We show by induction on $|\Lscr_1|$ that we can hit all cycles in $\Bscr_\text{ext}$ with some $M_\text{ext} \subseteq V(\Fscr)$ with $|M_\text{ext}| \leq |\Lscr_1|$:
  For $|\Lscr_1| = 1$ this is trivial.
  Otherwise, let $C_1 \in \Bscr_\text{ext}$ be minimal w.r.t.\@ $\inside$ and $C_2 \in \Fscr$ with some vertex $v \in V(C_1) \cap V(C_2)$ (cf.\@ Figure~\ref{fig:easy_structure_lemma}).
  Construct another laminar family $\Lscr^\prime$ by deleting all cycles inside $C_1$ and all cycles in $\Lscr \setminus \Fscr$ that contain $v$.
  Since $C_1$ contains some one-sided cycle, $\Lscr^\prime$ contains strictly less one-sided cycles and we can use the induction hypothesis on $\Lscr^\prime$.
  Also the deletion of cycles inside $C_1$ does (except for $C_1$ itself) not change $\Bscr_\text{ext}$ because $C_1 \notin \Bscr_\text{int}$ and $C_1$ was minimal.
  Thus, the induction hypothesis gives us a set $M^\prime_\text{ext} \subseteq V(\Fscr)$ that hits all cycles in $\Bscr_\text{ext} \cap \Lscr^\prime$ with $|M^\prime_\text{ext}| \leq |\Lscr_1| - 1$,
  so $M_\text{ext} := M^\prime_\text{ext} \cup \{v\}$ has the desired properties.
  
  Setting $M := M_\text{int} \cup M_\text{ext}$ yields a set as desired in the Lemma.
 \end{proof}

  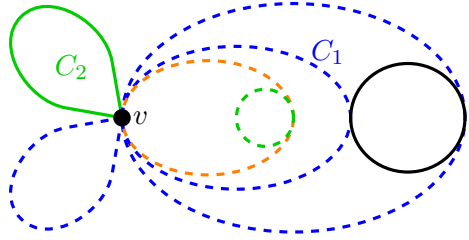
\begin{figure}[htb]
 \begin{center}
  \begin{tikzpicture}[scale=0.38, very thick]
   \draw[color=darkgreen] (0,0) -- (100:2) to[out=100,in=45] (135:5) to[out=225, in=170] (170:2) -- (0,0);
   \draw[color=blue, dashed] (0,0) -- (260:2) to[out=260,in=315] (225:5) to[out=135,in=190] (190:2) -- (0,0);
   \draw[color=blue, dashed] (6,0) ellipse (6 and 4);
   \draw[color=blue, dashed] (4,0) ellipse (4 and 2.5);
   \draw[color=orange, dashed] (3,0) ellipse (3 and 2);
   \draw[color=darkgreen, dashed] (5,0) ellipse (1 and 1);
   \draw[color=black] (10,0) ellipse (2 and 1.9);
   \draw[color=black, fill] (0,0) circle (0.25);
   \node[anchor=west] at (0,0) {$v$};
   \node[color=darkgreen] at (135:2.5) {$C_2$};
   \node[color=blue] at (7.2, 2.3) {$C_1$};
  \end{tikzpicture}
 \end{center}
 \caption{The cycles in $\Fscr$ are drawn in \textcolor{darkgreen}{green}, the cycles in $\Bscr_\text{ext}$ in \textcolor{blue}{blue}.
 Note that the orange cycle is not in $\Bscr_\text{ext}$ because it is in $\Bscr_\text{int}$.
 The inside of $C_1$ contains no other cycles in $\Bscr_\text{ext}$ and $C_1$ meets $C_2 \in \Fscr$ in $v$.
 We then add $v$ to our set $M_\text{ext}$ and recurse on the laminar family $\Lscr^\prime$ that is constructed by removing all cycles inside $C_1$ and all cycles that contain $v$ and are not in $\Fscr$.
 These cycles are drawn dashed.
 This step decreases the number of one-sided cycles.
 \label{fig:easy_structure_lemma}}
\end{figure}
 
 In the following, we assume $\Cscr$ to be an uncrossable family of cycles in a graph $G$, embedded in the sphere.
 We further assume $x \in \mathbb{R}^\Cscr$ to be a structured solution to the LP (\ref{eq:lp}) with support $\Lscr_x$.
 As in Theorem~\ref{thm:ub_4} we can assume $E(\Lscr_x)$ to be connected.
 Let $\Lscr_1 \subseteq \Lscr_x$ be the set of one-sided cycles in $\Lscr_x$.
 For each $0 \leq \alpha < 1$ we define $\Lalpha \subseteq \Lscr_1$ to be the set of one-sided cycles with LP value $>\alpha$ and set $r_\alpha := \frac{|\Lalpha|}{|\Lscr_1|}$.
 
 We will now give three possible choices for $\Fscr^*$ in Algorithm~\ref{alg:main_alg}.
 The first possibility is to choose a one-cardinality subset of $\Lscr_1$ as $\Fscr^*$, as in Theorem~\ref{thm:ub_4}.
 By Lemma~\ref{lemma:x_is_small_on_average} we directly get:
 
 \begin{lemma}\label{lemma:poss_1}
  There exists a cycle $C^* \in \Lscr_1$ with
  $x(\Nscr_{\Lscr_x}(C^*)) \leq 3 + \frac{x(\Lscr_1)}{|\Lscr_1|}$.
 \end{lemma}
 
 As a second possibility we define $\Fscr^*_\alpha := \Lalpha$ for any $\alpha \geq \frac{1}{2}$.
 Note that the cycles in $\Fscr^*_\alpha$ are pairwise vertex-disjoint because $x$ is a feasible LP solution.
 
 \begin{lemma}\label{lemma:poss_2}
  For any $\alpha \geq \frac{1}{2}$ we have
  \begin{equation*}\frac{x\left(\bigcup_{C^\prime \in \Fscr^*_\alpha}\Nscr_{\Lscr_x}(C^\prime)\right)}{|\Fscr^*_\alpha|} \leq 1 + \frac{1-\alpha}{r_\alpha}
  \end{equation*}
 \end{lemma}
 \begin{proof}
  Let $\alpha \geq \frac{1}{2}$.
  By Lemma~\ref{lemma:neighbourhood_of_large_set} there is a set $M \subseteq V(\Fscr^*_\alpha)$ with $|M| \leq |\Fscr^*_\alpha| + |\Lscr_1|$ such that each cycle in $\bigcup_{C^\prime \in \Fscr^*_\alpha}\Nscr_{\Lscr_x}(C^\prime)$ contains a vertex of $M$.
  Now
  \begin{align*}
   & \ x\left(\bigcup_{C^\prime \in \Fscr^*_\alpha}\Nscr_{\Lscr_x}(C^\prime)\right) \\
   \leq & \ \sum_{v \in M} x(\{C \in \Lscr_x \setminus \Fscr^*_\alpha : v \in V(C)\}) + x(\Fscr^*_\alpha) \\
   \leq & \ |\Fscr^*_\alpha| + (1-\alpha)|\Lscr_1|
  \end{align*}
  holds, where the last inequality follows from the fact that there are $|\Fscr^*_\alpha|$ vertices in $M$ covering $\Fscr^*_\alpha$, and on the other vertices we only have to count the LP value of cycles not in $\Fscr^*_\alpha$.
 \end{proof}
 
 As a third possibility we take a look at the sets $\Lalpha$ with $\frac{1}{4} \leq \alpha < \frac{1}{2}$.
 For these we know that at most three cycles in $\Lalpha$ can share a vertex.
 Let $G^\prime$ be the \emph{conflict graph} for the cycles in $\Lalpha$; i.e.\@ $G^\prime$ is the graph on vertex set $\Lalpha$ such that two cycles in $\Lalpha$ are connected by an edge in $G^\prime$ if and only if they share a vertex in $G$.
 Since each vertex is contained in at most three cycles of $\Lalpha$, $G^\prime$ is planar.
 Furthermore, the cycles in $\Lscr_1^{>1-\alpha} \subseteq \Lalpha$ correspond to isolated vertices in $G^\prime$.
 By the Four Colour Theorem~\cite{AppH76} we can partition $V(G^\prime) - \Lscr_1^{>1-\alpha}$ into four stable sets.
 The largest of those, together with 
 $\Lscr_1^{>1-\alpha}$, yields a stable set in $G^\prime$ of size at least $|\Lscr_1^{>1-\alpha}| + \frac{1}{4}(|\Lalpha|- |\Lscr_1^{>1-\alpha}|)$.
 We let $\Fscr^*_\alpha$ be the set of cycles in $\Lalpha$ corresponding to such a stable set in $G^\prime$.
 
 \begin{lemma}\label{lemma:poss_3}
  For any $\frac{1}{4} \leq \alpha < \frac{1}{2}$ we have
  \begin{equation*}\frac{x\left(\bigcup_{C^\prime \in \Fscr^*_\alpha}\Nscr_{\Lscr_x}(C^\prime)\right)}{|\Fscr^*_\alpha|} \leq 1 + \frac{4(1-\alpha)}{r_\alpha + 3 r_{1-\alpha}}
  \end{equation*}
 \end{lemma}
 \begin{proof}
  Let $\alpha \geq \frac{1}{2}$.
  By a similar argument as in Lemma~\ref{lemma:poss_2} we get
  \begin{equation*}
   \frac{x\left(\bigcup_{C^\prime \in \Fscr^*_\alpha}\Nscr_{\Lscr_x}(C^\prime)\right)}{|\Fscr^*_\alpha|}
   \leq 1 + \frac{(1-\alpha)|\Lscr_1|}{|\Fscr^*_\alpha|}
  \end{equation*}
  Inserting the bound $|M^*_\alpha| \geq |\Lscr_1^{>1-\alpha}| + \frac{1}{4}(|\Lalpha|- |\Lscr_1^{>1-\alpha}|)$ yields the lemma.
 \end{proof}
 
 One of these possibilities for $\Fscr^*$ will be sufficient to prove the desired upper bound of $\frac{20 + \sqrt{130}}{9}$ for the integrality gap:
 
 \begin{lemma}\label{lemma:one_poss_is_good_enough}
  There exists a set $\Fscr^* \subseteq \Lscr_1$ with
  $\frac{x\left(\bigcup_{C^\prime \in \Fscr^*}\Nscr_{\Lscr_x}(C^\prime)\right)}{|\Fscr^*|} \leq \frac{20 + \sqrt{130}}{9}$
 \end{lemma}
 \begin{proof}
  Define $\beta := \frac{20 + \sqrt{130}}{9}$.
  We will either pick one of the sets $\Fscr^*_\alpha$ for $\frac{1}{4} \leq \alpha < 1$ from Lemma~\ref{lemma:poss_2} or Lemma~\ref{lemma:poss_3} or we will use $\Fscr^* := \{C^*\}$ with the cycle $C^*$ from Lemma~\ref{lemma:poss_1}.
  Assume none of these sets $\Fscr^*$ fulfills the above inequality.
  Then Lemma~\ref{lemma:poss_2} implies
  \begin{align}
   & 1 + \frac{1 - \alpha}{r_\alpha} > \beta \\
   \Leftrightarrow \quad & r_\alpha < \frac{1 - \alpha}{\beta - 1}
  \end{align}
 for all $\alpha \geq \frac{1}{2}$.
 For $\frac{1}{4} \leq \alpha < \frac{1}{2}$, Lemma~\ref{lemma:poss_3} yields
 \begin{align}
  & 1 + \frac{4(1-\alpha)}{r_\alpha + 3 r_{1-\alpha}} > \beta \\
  \Leftrightarrow \quad & r_\alpha + 3r_{1-\alpha} < \frac{4(1-\alpha)}{\beta - 1}
 \end{align}
 Furthermore, we have
 \begin{equation*}
  x(\Lscr_1) = \sum_{C \in \Lscr_1} \int_0^1 \mathbb{1}_{x(C) > \alpha} d\alpha = \int_0^1 \sum_{C \in \Lscr_1} \mathbb{1}_{x(C) > \alpha} d\alpha = |\Lscr_1| \int_0^1 r_\alpha d\alpha
 \end{equation*}
 Thus, Lemma~\ref{lemma:poss_1} implies
 \begin{equation}
  \beta < 3 + \int_0^1 r_\alpha d\alpha
 \end{equation}
 For any threshold $0 < \delta \leq \frac{1}{6}$ that might be chosen later we can bound this as follows, using the fact that the $r_\alpha$ are non-increasing:
 \begin{align*}
    \beta & < \ 3 + \int_0^1 r_\alpha d\alpha \\
  & = \ 3 + \int_0^{\frac{1}{2} - \delta} r_\alpha d\alpha + \int_{\frac{1}{2} - \delta}^{\frac{1}{2} + 3 \delta} r_\alpha d\alpha + \int_{\frac{1}{2} + 3 \delta}^1 r_\alpha d\alpha \\
  & \leq \ 3 + \int_0^{\frac{1}{2} - \delta} r_\alpha d\alpha + \int_{\frac{1}{2} - \delta}^{\frac{1}{2}} r_\alpha d\alpha + 3 \int_{\frac{1}{2}}^{\frac{1}{2} + \delta} r_\alpha d\alpha + \int_{\frac{1}{2} + 3 \delta}^1 r_\alpha d\alpha \\
  & = \ 3 + \int_0^{\frac{1}{2} - \delta} r_\alpha d\alpha + \int_{\frac{1}{2} - \delta}^{\frac{1}{2}} r_\alpha + 3 r_{1 - \alpha} d\alpha  + \int_{\frac{1}{2} + 3 \delta}^1 r_\alpha d\alpha \\
  & \leq \ 3 + \int_0^{\frac{1}{2} - \delta} 1 d\alpha + \int_{\frac{1}{2} - \delta}^{\frac{1}{2}} \frac{4 (1 - \alpha)}{\beta - 1} d\alpha + \int_{\frac{1}{2} + 3 \delta}^1 \frac{1-\alpha}{\beta - 1} d\alpha \\
  & = \ \frac{7}{2} - \delta + \frac{2 \delta (\delta + 1)}{\beta - 1} + \frac{(1-6 \delta)^2}{8 (\beta - 1)}
 \end{align*}
 This term attains its minimum of
 $\frac{\beta^2 - 94 \beta + 90}{26 (1 - \beta)}$ at $\delta := \frac{2 \beta - 3}{26} < \frac{1}{6}$, so we get
 \begin{align*}
  \beta < \frac{\beta^2 - 94 \beta + 90}{26 (1 - \beta)}
 \end{align*}
 This is a contradiction for $\beta = \frac{20 + \sqrt{130}}{9}$.
 \end{proof}
 
 As an immediate consequence we get our main theorem, similar to Theorem~\ref{thm:ub_4}:
 
 \begin{theorem}\label{thm:best_ub}
  Let $G$ be a planar graph, embedded in the sphere, and $\Cscr$ an uncrossable family of cycles in $G$.
  Then there exists an integral solution to the vertex-disjoint cycle packing LP with at least $\frac{9}{20 + \sqrt{130}}$ the LP value.
  If $\Cscr$ is given by a weight oracle we can compute such a solution in polynomial time.
 \end{theorem}
 \begin{proof}
  As in Theorem~\ref{thm:ub_4} we apply Algorithm~\ref{alg:main_alg}.
  In contrast to the procedure in Theorem~\ref{thm:ub_4} however we use a set $\Fscr^*$ as guaranteed by Lemma~\ref{lemma:one_poss_is_good_enough} in step~\ref{step:choose_cycle_set} of the algorithm instead of a set consisting of just one one-sided cycle.
  Thus, in each step we increase the number of cycles with LP value 1 by $|\Fscr^*|$ while decreasing the LP value on $\Lscr_x$ by at most $\frac{20 + \sqrt{130}}{9} |\Fscr^*|$.
  Therefore we arrive at an integral solution to the LP with at least $\frac{9}{20 + \sqrt{130}}$ the LP value.
  
  Note that a set $\Fscr^*$ as in Lemma~\ref{lemma:one_poss_is_good_enough} can be found in polynomial time in $|\Lscr_x|$:
  For the cycle guaranteed by Lemma~\ref{lemma:poss_1} we can try all one-sided cycles.
  For the sets $\Fscr^*_\alpha$ used in Lemma~\ref{lemma:poss_2}, note that there are only linearly many different sets $\Lalpha$ to consider.
  The sets $\Fscr^*_\alpha$ used in Lemma~\ref{lemma:poss_3} are constructed from $\Lalpha$ by applying the Four Colour Theorem, which can also be done in polynomial time \cite{RobSST96}.
 \end{proof}
 
 \subsection{The edge-disjoint version}\label{sec:edge_disjoint}
 
 This section is dedicated to proving an edge-disjoint version of Theorem~\ref{thm:best_ub}.
 One possibility to do this is to give edge-disjoint versions of Algorithm~\ref{alg:main_alg}, the Structure Lemma~\ref{lemma:structure_lemma_using_incidences} as well as Lemma~\ref{lemma:neighbourhood_of_large_set} and then Lemma~\ref{lemma:one_poss_is_good_enough}.
 All of this is possible analogous to the vertex-disjoint versions, however we can also use a simple reduction by Schlomberg, Thiele and Vygen~\cite{SchTV22}.
 This does not generally reduce edge-disjoint cycle packing to vertex-disjoint cycle packing in general, but it does so for cycle packing in laminar cycle families:
 
 \begin{lemma}\label{lemma:reduction_edge_to_vertex}
  Given a planar graph $G$, embedded in the sphere, and a laminar family $\Lscr$ of cycles in $G$, we can compute in polynomial time a planar graph $G^\prime$ and a laminar family $\Lscr^\prime$ of cycles in $G^\prime$, together with a bijection $f \colon \Lscr \to \Lscr^\prime$ such that for any $C_1, C_2 \in \Lscr$ we have that $E(C_1) \cap E(C_2) = \emptyset$ if and only if $V(f(C_1)) \cap V(f(C_2)) = \emptyset$.
 \end{lemma}
 \begin{proof}
  Define $V(G^\prime) := E(G)$.
  For any path $P=e_1 e_2$ of length two in a cycle $C \in \Lscr$ we add the edge $e_C^P := \{e_1, e_2\}$ to $E(G^\prime)$.
  For any $C \in \Lscr$ let $f(C)$ be the cycle consisting of all edges $e_C^P \in E(G^\prime)$ for any path $P$ of length two inside $C$.
  
  Since $\Lscr$ is laminar, $G^\prime$ can be embedded planarly such that $\Lscr^\prime := \{f(C) : C \in \Lscr\}$ defines a laminar family of cycles, as shown in Figure~\ref{fig:reduction_edge_to_vertex}.
  By definition, all cycles in $\Lscr^\prime$ are edge-disjoint and two cycles $C_1, C_2 \in \Lscr$ share an edge in $G$ if and only if $f(C_1)$ and $f(C_2)$ share a vertex in $G^\prime$.
 \end{proof}
 
 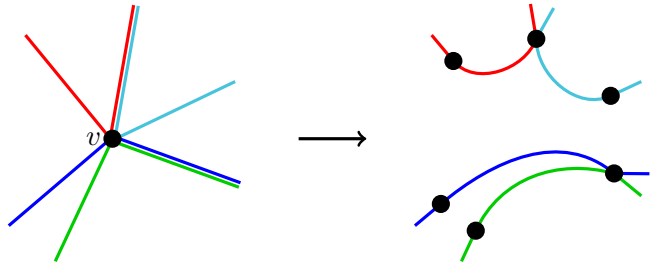
\begin{figure}[htb]
 \begin{center}
  \begin{tikzpicture}[scale=0.45, very thick]
   \draw[color=red] (130:4) -- (170:0.07) -- (81:4);
   \draw[color=SkyBlue] (79:4) -- (-10:0.07) -- (25:4);
   \draw[color=blue] (-140:4) -- (70:0.07) -- (-19:4);
   \draw[color=darkgreen] (-115:4) -- (-110:0.07) -- (-21:4);
   \node[draw, fill, circle, inner sep=2pt, color=black] at (0,0) {};
   \node[anchor=east] at (0,0) {$v$};
   \draw[->] (5.5,0) -- (7.5,0);
   \begin{scope}[shift={(12,0)}]
    \draw[color=red] (130:4) -- (130:3) to[out=-50,in=-110] (80:3) -- (85:4);
    \draw[color=SkyBlue] (75:4) -- (80:3) to[out=-90,in=-155] (25:3) -- (25:4);
    \draw[color=blue] (-140:4) -- (-140:3) to[out=40,in=140] (-20:3) -- (-15:4);
    \draw[color=darkgreen] (-115:4) -- (-115:3) to[out=65,in=165] (-20:3) -- (-25:4);
    \node[draw, fill, circle, inner sep=2pt, color=black] at (130:3) {};
    \node[draw, fill, circle, inner sep=2pt, color=black] at (80:3) {};
    \node[draw, fill, circle, inner sep=2pt, color=black] at (25:3) {};
    \node[draw, fill, circle, inner sep=2pt, color=black] at (-140:3) {};
    \node[draw, fill, circle, inner sep=2pt, color=black] at (-115:3) {};
    \node[draw, fill, circle, inner sep=2pt, color=black] at (-20:3) {};
   \end{scope}
  \end{tikzpicture}
 \end{center}
 \caption{Example for the construction of $G^\prime$ and $\Lscr^\prime$:
 The left picture shows four cycles in $\Lscr$ containing the vertex $v$.
 The six edges incident to $v$ correspond to vertices of $G^\prime$, as shown in the right picture.
 Since $\Lscr$ is laminar, the paths of length two in cycles of $\Lscr$ can be embedded planarly as edges of $G^\prime$. Cycles in $\Lscr$ share an edge if and only if the corresponding cycles in $\Lscr^\prime$ share a vertex.
 \label{fig:reduction_edge_to_vertex}}
\end{figure}
 
 Using this reduction, we can easily extend Theorem~\ref{thm:best_ub} to the edge-disjoint case:
 
 \begin{theorem}\label{thm:best_ub_edge_disjoint}
  Let $G$ be a planar graph, embedded in the sphere, and $\Cscr$ an uncrossable family of cycles in $G$.
  Then there exists an integral solution to the edge-disjoint cycle packing LP (\ref{eq:lp_edgedisjoint}) a with at least $\frac{9}{20 + \sqrt{130}}$ the LP value.
  If $\Cscr$ is given by a weight oracle we can compute such a solution in polynomial time.
 \end{theorem}
 \begin{proof}
  We first apply Lemma~\ref{lemma:laminar_lp_solution} to get an optimum LP solution $x$ to (\ref{eq:lp_edgedisjoint}) with laminar support $\Lscr$.
  We then apply Lemma~\ref{lemma:reduction_edge_to_vertex} to get a laminar set $\Lscr^\prime$ of cycles in a planar graph $G^\prime$ with a bijection $f \colon \Lscr \to \Lscr^\prime$ such that edge-disjointness in $\Lscr$ translates to vertex-disjointness in $\Lscr^\prime$.
  
  Note that $y \in \mathbb{R}^{\Lscr^\prime}$ with $y_{f(C)} = x_C$ for all $C \in \Lscr$ defines a feasible solution to the LP (\ref{eq:lp}) on $\Lscr^\prime$.
  Similar to Theorem~\ref{thm:best_ub} we can find an integral solution $\bar{y} \in \mathbb{R}^{\Lscr^\prime}$ to (\ref{eq:lp}) on $\Lscr^\prime$ with $\bar{y}(\Lscr^\prime) \geq \frac{9}{20 + \sqrt{130}} y(\Lscr^\prime)$.
  Setting $\bar{x}_C := \bar{y}_{f(C)}$ for all $C \in \Lscr$ then yields an integral solution to (\ref{eq:lp_edgedisjoint}) on $\Lscr$ with $\bar{x}(\Lscr) = \bar{y}(\Lscr^\prime) \geq \frac{9}{20 + \sqrt{130}} y(\Lscr^\prime) = \frac{9}{20 + \sqrt{130}} x(\Lscr)$.
 \end{proof}
 
 \section{Acknowledgement}
 Thanks to Luise Puhlmann and Jens Vygen for listening, reading and improvement ideas.
 
 \bibliographystyle{plain}
\bibliography{bibliography} 
 
 \end{document}